\documentclass{amsart}[12pt]
\usepackage{amssymb}
\usepackage{mathrsfs}
\usepackage{amsfonts}
\usepackage{amsmath}
\usepackage{pxfonts}
\usepackage{amssymb}
\usepackage{enumitem}
\usepackage{amsfonts}
\usepackage{graphicx}
\usepackage{texdraw}
\usepackage{graphpap}
\usepackage{wasysym}
\usepackage{color}
\usepackage{pxfonts}
\usepackage[OT2,T1]{fontenc}

\input txdtools

\newtheorem{thm}{Theorem}[section]
\newtheorem{cor}[thm]{Corollary}
\newtheorem{lem}[thm]{Lemma}

\theoremstyle{definition}

\theoremstyle{remark}
\newtheorem*{rem}{Remark}

\numberwithin{equation}{section}

\newcommand{\refb}{\operatorname{ref,bulk}}
\newcommand{\refbo}{\operatorname{ref,boundary}}

\newcommand{\Qeff}{Q_{\mathrm{eff}}}

\newcommand{\Prob}{{\mathbb{P}}}

\newcommand{\calW}{{\mathscr W}}

\newcommand{\Int}{\operatorname{Int}}

\newcommand{\erfc}{\operatorname{erfc}}

\newcommand{\R}{{\mathbb R}}

\newcommand{\C}{{\mathbb C}}

\newcommand{\E}{{\mathbb E}}

\newcommand{\eps}{{\varepsilon}}

\newcommand{\re}{\operatorname{Re}}
\newcommand{\im}{\operatorname{Im}}

\renewcommand{\d}{{\partial}}
\newcommand{\dbar}{\bar{\partial}}
\newcommand{\1}{\mathbf{1}}

\newcommand{\bigO}{\mathcal{O}}

\newcommand{\dist}{\operatorname{dist}}
\newcommand{\supp}{\operatorname{supp}}

\renewcommand{\P}{\Prob}

\newcommand{\ET}[1]{{\color{blue}{#1}}}

\makeatletter
\newcommand*\bigcdot{\mathpalette\bigcdot@{.5}}
\newcommand*\bigcdot@[2]{\mathbin{\vcenter{\hbox{\scalebox{#2}{$\m@th#1\bullet$}}}}}
\makeatother
\begin{document}

\title[One-point densities of Hele-Shaw $\beta$-ensembles]{Remarks on the one-point density of Hele-Shaw $\beta$-ensembles}

\author{Yacin Ameur}
\address{Yacin Ameur\\
Department of Mathematics\\
Lund University\\
22100 Lund, Sweden}
\email{ Yacin.Ameur@math.lu.se}

\author{Erik Troedsson}
\address{Erik Troedsson\\
Department of Mathematics\\
Lund University\\
22100 Lund, Sweden}
\email{Erik.Troedsson@math.lu.se}

\begin{abstract} In this note we prove equicontinuity for the family of one-point densities with respect to a two-dimensional Coulomb gas at an inverse temperature $\beta\ge 1/2$ confined by an external potential
of Hele-Shaw (or quasi-harmonic) type. As a consequence, subsequential limiting Lipschitz continuous densities are defined on the microscopic scale. There are several additional results, for example comparing the one-point density with the thermal equilibrium density.
\end{abstract}

\subjclass[2010]{82B21; 60K35}

\keywords{$\beta$-ensemble; Hele-Shaw potential; one-point density; equicontinuity; Ward's identity; thermal equilibrium measure}

\maketitle

\section{Introduction and main results} The one-point function (or better: one-point measure) is a fundamental construct in Coulomb gas theory. Indeed, it is commonly believed that this single object
should in some sense contain all essential information about the gas.

In this note we consider a Coulomb gas ensemble $\{z_j\}_1^n$ in the complex plane $\C$, at an inverse temperature $\beta>0$, confined by an external potential of  the form
\begin{equation}\label{hsdef}Q(z)=\begin{cases}\Delta \cdot |z|^2+h(z),& z\in\Sigma\cr
+\infty,& \text{otherwise}\end{cases},
\end{equation} where $\Sigma$ is a closed subset of $\C$ (of positive capacity), $h(z)$ is some function harmonic in a neighbourhood of $\Sigma$ and $\Delta>0$ is a constant.

It is convenient to take $\Sigma$ to be compact, so that the normalized area
$|\Sigma|:=\int_\Sigma\, dA$
is finite, where we write $dA(z)=\frac 1 \pi dxdy$, $(z=x+iy)$ for the normalized Lebesgue measure on $\C$.

By Frostman's theorem (see \cite{ST}) there is a unique equilibrium measure $\sigma=\sigma[Q]$, which minimizes the weighted energy
\begin{equation}\label{wen}I_Q[\mu]=\int_{\C^2}\log\frac 1 {|z-w|}\, d\mu(z)\, d\mu(w)+\mu(Q)\end{equation}
over all compactly supported Borel probability measures $\mu$. (Here and throughout, we write $\mu(f)$ for $\int f\, d\mu$.)

Intuitively, the measure $\sigma$ represents the way a unit blob of charge will organize itself if
exposed to the external potential $Q$ for a long time.

The support
$$S=S[Q]:=\supp\sigma$$
is called the droplet in potential $Q$.
In what follows, we will assume that $S$ is contained in the interior of $\Sigma$.

As is well known \cite{ST}, the equilibrium measure is then absolutely continuous and has the structure
$$d\sigma(z)=\Delta\cdot \1_S(z)\, dA(z).$$

Moreover, by virtue of Sakai's theorem (cf. \cite{Sa,LM} and the references there) the boundary $\d S$ consists of finitely many Jordan curves, which are analytic except possibly for finitely many singular points which are either contact points or certain types of conformal cusps.

Potentials of the above type \eqref{hsdef} are sometimes
termed ``Hele-shaw potentials'' or ``quasi-harmonic potentials'', the model case being the quadratic (or ``harmonic'') potential $Q=|z|^2$. The truncation that $Q=+\infty$ near infinity is convenient and fairly standard
since \cite{EF}; the precise details of the hard wall outside $\Sigma$ is largely irrelevant for the statistics as long as $S$ is contained in
the interior of $\Sigma$.

Now fix a large integer $n$ and a parameter $\beta>0$ and consider the Boltzmann-Gibbs measure on $\C^n$
\begin{equation}\label{bgm}d\Prob_n^\beta(z_1,\ldots,z_n)=\frac 1 {Z_n^\beta}e^{-\beta H_n(z_1,\ldots,z_n)}\, dA_n(z_1,\ldots,z_n),\end{equation}
where $Z_n^\beta$ is the normalizing constant and $H_n$ is the Hamiltonian,
\begin{equation}\label{hamn}H_n=\sum_{i\ne j}\log \frac 1 {|z_i-z_j|}+n\sum_{i=1}^n Q(z_i).\end{equation}
Also we write $dA_n$ for the product measure $(dA)^{\otimes n}$.

The Coulomb gas (or $\beta$-ensemble) $\{z_j\}_1^n$ in external potential $Q$ is a configuration picked randomly with respect to \eqref{bgm}.\footnote{Our normalization is chosen so that $\beta=1$ corresponds to the determinantal case. A more common convention replaces our ``$\beta$'' by ``$\frac \beta 2$'', and then $\beta=2$ means the determinantal case.}

The 1-point function of \eqref{bgm} (in background measure $dA$), denoted $R_n^\beta(p)$, is
$$R_n^\beta(p)=\lim_{\eps\to 0}\frac 1 {\eps^2}\E_n^\beta[N(p,\eps)],$$
where $N(p,\eps)$ denotes the number of particles in the configuration $\{z_j\}_1^n$ which fall in the disk $D(p,\eps)$ with center $p$ and radius $\eps$.

Let $\check{Q}(z)$ (the obstacle function) be the pointwise supremum of $s(z)$ where $s$ ranges through the class of subharmonic functions on $\C$ such that $s(w)\le Q(w)$ for all $w\in\C$ and $s(w)\le 2\log|w|+\bigO(1)$ as $w\to\infty$.

The coincidence set $S^*:=\{Q=\check{Q}\}$ is then a compact set containing $S$ and possibly some ``shallow points'', such that $\sigma(S^*\setminus S)=0$. 
We shall here always assume that $S^*=S$.

We will write
\begin{equation}\label{qeff}
	\Qeff:=Q-\check{Q}
\end{equation}
for the ``effective potential''.

It is convenient to fix two real parameters $\Delta_0>0$ and $\eta_0>0$ (both ``small'') and assume that
\begin{equation}\label{lower_bound}\Delta\ge \Delta_0,\qquad \text{and}\qquad \dist(S,\C\setminus \Sigma)\ge 2 \eta_0.\end{equation}

We have the following theorem, which combines and improves on the bounds in \cite[Theorem 2]{Th} and \cite[Theorem 1]{A} for the given class of potentials.

\begin{thm}[Upper bound]\label{loca}
There is a constant $C=C(\Delta_0,\eta_0)$ such that for all $z\in\C$,
\begin{equation}\label{rbound}R_n^\beta(z)\le C^\beta (1+\beta^{-2})n\Delta \min\{1,ne^{-n\beta \Qeff(z)}\}.\end{equation}
\end{thm}

In case the boundary $\d S$ is everywhere regular (i.e. real-analytic) we can obtain a more explicit estimate in \eqref{rbound}. Indeed, a Taylor expansion as in the proof of \cite[Lemma 2.1]{A} gives
$$\Qeff(z)=(2\Delta)\delta_S(z)^2+\bigO(\delta_S(z)^3),\qquad (\delta_S(z):=\dist(z,S)\to 0).$$
The behaviour of $\Qeff$ close to a cusp is a bit more involved, see \cite[Section 2.2]{AKMW} for relevant asymptotics.

\begin{rem} (i) While \eqref{rbound} suffices for our present purposes, we believe that the following, slightly better bound should be true:
\begin{equation}\label{murkla}R^\beta_n(z)\le C^\beta (1+\beta^{-2})n\Delta e^{-n\beta \Qeff(z)}.\end{equation}
The bound \eqref{murkla} is well-known when $\beta=1$, as it follows from the identity $$R_n^{(\beta=1)}(z)=
 \sup\{|f(z)|^2\,;\, f\in\calW_n,\,\|f\|\le 1\}$$ where $\calW_n$ is the space of weighted polynomials introduced in Section \ref{local}.
An application of Lemma \ref{silem} implies \eqref{murkla} for $\beta=1$ (with a much better constant).

(ii) To prove \eqref{rbound}, we will first prove the uniform estimate $R_n^\beta(z)\le C^\beta (1+\beta^{-2})n\Delta$ for $z$ in some neighbourhood of the droplet. In this case, our analysis
is based on the multiscale approach 
by E.~Thoma in \cite{Th}. We remark that the obtained constant $C=C(\Delta_0,\eta_0)$ is very large. We shall then prove the global estimate
$R_n^\beta(z)\le e^\beta n^2\Delta  e^{-n\beta \Qeff(z)}$ by using ideas from \cite{A}. In particular the estimate \eqref{rbound} is improved outside of a vicinity of the droplet.
\end{rem}

We also have the following main result.

\begin{thm}[Equicontinuity]\label{mth}
Under the assumptions of Theorem \ref{loca} we have for all $z,w$ in the interior of $\Sigma$ with distance to the boundary at least $\eta_0$, and all $\beta\ge 1/2$, that
$$|R_n^\beta(z)-R_n^\beta(w)|\le C^\beta n\Delta \sqrt{n\Delta}|z-w|.$$
\end{thm}

We turn to consequences of our above results.

Fix a point $p$ in the interior of $\Sigma$ and consider the rescaled process $\{w_i\}_1^n$ where $$w_i=\sqrt{n\Delta}\cdot (z_i-p),$$
and where $\{z_i\}_1^n$ (as always) denotes a random sample from \eqref{bgm}.

 The one-point function of $\{w_i\}_1^n$ is
\begin{equation}\label{rescp}\rho_n^\beta(z)=\rho_n^{\beta,p}(z):=\frac 1 {n\Delta}R_n^\beta(p+\frac z {\sqrt{n\Delta}}).\end{equation}
By Theorem \ref{mth} we have the uniform Lipschitz continuity
$$|\rho_n^\beta(z)-\rho_n^\beta(w)|\le C^\beta|z-w|,$$
where $z,w$ belongs to an arbitrary but fixed compact subset of $\C$ and $n\ge n_0(p)$.

As a consequence, the family $\{\rho_n^\beta\}$ is uniformly equicontinuous on compact subsets of $\C$. The family is also uniformly bounded by Theorem \ref{loca}. In view of the Arzel\`{a}-Ascoli theorem, we obtain the following corollary.

\begin{cor} \label{bracc} If $\beta\ge 1/2$, then each subsequence of $\{\rho_n^\beta\}$ contains a further subsequence converging uniformly on compact sets to a Lipschitz continuous limiting one-point function $\rho^\beta=\rho^{\beta,p}$ satisfying
$$|\rho^\beta(z)-\rho^\beta(w)|\le C^\beta|z-w|.$$
\end{cor}

Note that it follows from Theorem \ref{loca} (or \cite[Theorem 1]{A}) that if we consider a sequence $p_n$ of zooming points such that
$\sqrt{n\log n}\,\delta_S(p_n)\to\infty$ as $n\to\infty,$
then the rescaled densities $\rho_n^{\beta,p_n}$ converge to zero uniformly on compact subsets of $\C$. In other words, the rescaled processes
$\{w_j\}_1^n$ converge on compact subsets, as $n\to\infty$, to the empty point-process.

\begin{rem} The uniqueness of limit $\rho^\beta$ remains an open question even in the bulk. See Section \ref{concrem} for a discussion.
\end{rem}

\begin{rem} Our choice of working with Hele-Shaw potentials is mainly for convenience, since it simplifies several arguments. However, with some extra effort, it is possible to extend our
results to situations where the Laplacian of $Q$ can vary, but is bounded above and below by positive constants
 in a neighbourhood of the droplet.
\end{rem}

\begin{rem} We stress that the precise choice of cut-off set $\Sigma$ in \eqref{hsdef} is irrelevant as long as it contains the droplet $S$ in its interior; all the interesting statistics takes place in the vicinity of the droplet, and a change in $\Sigma$ will have a negligible effect for any reasonable statistic associated with the Coulomb gas. \footnote{The case when $\d\Sigma$ intersects the droplet is known as ``hard edge'' and is not considered in the present work. In this case $\Sigma$ may have a drastic effect for the Coulomb gas, see \cite{ACCH} and references therein.}

In fact, our results remain valid (with minor modifications of the proofs) provided that $Q(z)$ obeys the growth condition
$Q(z)-2\log|z|\to+\infty$ as $z\to\infty$, which is just strong enough to ensure that a version of Frostman's theorem holds \cite{ST}. Unfortunately, this growth condition is quite restrictive for potentials with constant Laplacian; for example it fails for the cubic potential $Q=|z|^2+c\re (z^3)$. For such reasons it is convenient to impose the cut-off \eqref{hsdef}. However, when we speak, for example, of the $\beta$-Ginibre ensemble, there is no problem in picking $Q(z)=|z|^2$ globally. The same holds for the almost circular $\beta$-ensemble below.
\end{rem}

\subsection{The almost circular $\beta$-ensemble} 
The one-dimensional circular $\beta$-ensemble, C$\beta$E, is associated with much work and very precise results, depending among other things on the evaluation of Selberg integrals. (See eg \cite{Fo} or \cite{FW} and the references there as well as the recent work \cite{CN}; see also \cite{M} with respect to CUE.)

A scale of two-dimensional ensembles where the droplet is an $n$-dependent thin annulus which collapses to the unit circle at a suitable rate was introduced in the works \cite{AB,BS,FBKSZ}. We will now exploit the form of our above estimates, specifically their dependence on $n$ and $\Delta$, to draw conclusions about such
``almost circular'' situations. We shall obtain the existence of new types of subsequential scaling limits, which might be of some intrinsic interest.

By definition (cf. \cite[Section 1.6]{AB}), the almost circular induced $\beta$-Ginibre ensemble
corresponds to an $n$-dependent
(Hele-Shaw) potential of the form
$$Q(z)=\Delta_n\cdot |z|^2
+b_n\log \frac 1 {|z|},$$
where we fix a ``non-circularity parameter'' $s>0$ and define $$\Delta_n=\Delta_{n,s}:=\frac n {s^2}$$
and choose 
$b_n$ so that the function $g_n(r)=\Delta_n\cdot r^2+b_n\log \frac 1 r$ satisfies 
$g_n'(1)=2$; (compare Section 1.6 in \cite{AB}).

The condition $g_n'(1)=2$ ensures that the outer boundary of the droplet coincides with the unit circle, and a simple computation shows that
the droplet $S_n=S[Q_n]$ is the thin annulus $$S_n=\left\{z\,;\,\sqrt{1-\frac {s^2}n}\le |z|\le 1\right\}.$$
A random sample $\{z_j\}_1^n$ from such an ensemble with $\beta=1$ is depicted in Figure \ref{fig_inducedG}.

\begin{figure}[h]
	\begin{center}
		\includegraphics[width=0.6\textwidth]{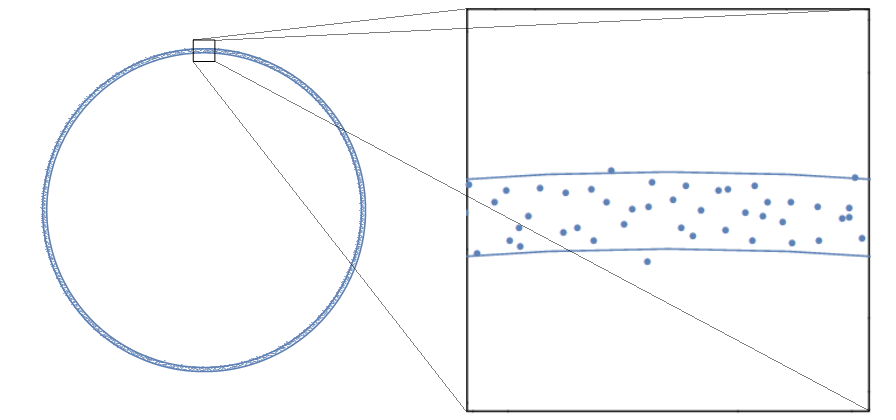}	
	\end{center}
	\caption{A sample from an almost circular Ginibre ensemble at $\beta=1$.}
	\label{fig_inducedG}
\end{figure}

It is natural to fix a point $p_n$ on the circle at the ``middle'' of the annulus:
$$p_n=\frac 1 2 (1+\sqrt{1-\frac {s^2}n})e^{i\alpha},$$
where $\alpha\in\R$ is an arbitrary but fixed angle.
We rescale about $p_n$ using the map
$$T_n(z):= -ie^{i\alpha}(n/s)(z-p_n).$$

We denote by $\{w_j\}_1^n$ the rescaled sample, i.e., we put $w_j=T_n(z_j)$. For large $n$, the system $\{w_j\}_1^n$ tends to occupy a horizontal strip about the real axis, see Figure \ref{fig_inducedG}.

By Corollary \ref{bracc}, the family of 1-point densities $\rho_n^\beta=\rho_n^{\beta,p_n}$ of $\{w_j\}_1^n$ is precompact in the appropriate sense, and we can extract Lipschitzian subsequential limits
$\rho^\beta=\lim \rho^\beta_{n_k}$.

Moreover, taking in consideration the rotational symmetry of the original ensemble, it is easy to see that each subsequential limit $\rho^\beta$ is \textit{horizontal translation invariant}, i.e.,
$$\rho^\beta(w+x)=\rho^\beta(w)$$
for all $w\in\C$ and $x\in\R$.

In \cite[Theorem 1.9]{AB} it is shown that $\rho^{(\beta=1)}(w)=F_s(2\im w)$ where
$$F_s(z)=\frac 1 {\sqrt{2\pi}}\int_{-s/2}^{\,s/2}e^{-\frac 1 2(z-t)^2}\, dt.$$
As far as we are aware, this one-point function was discovered in the 1990s in connection with almost-Hermitian random matrices from suitably scaled elliptic Ginibre ensembles, see \cite{FKS,Fo}, cf.~
\cite{ACV,AB} and the reference therein for more on this development. We here note that the short derivation of the one-point function in \cite[Theorem 1.9]{AB} uses Ward's equation \eqref{topp} and some simple apriori properties
of a subsequential limit $\rho^{(\beta=1)}$.

When $\beta=1$, the infinite point-process defined by the above 1-point function interpolates between the one-dimensional sine-process (as $s\to 0$) and the infinite Ginibre ensemble (as $s\to\infty$), see \cite{ACV,AB}. In the case of a general $\beta$, it seems natural
to ask to what extent a similar interpolation takes place.
(This question is not as straightforward as it might seem, since the uniqueness of an infinite $\beta$-Ginibre process is not clear when $\beta\ne 1$. See Section \ref{concrem}.)

\subsubsection*{Further results} In Section \ref{concrem} we give further remarks concerning (subsequential) microscopic densities and their properties. In particular we
recall the limit Ward equation from \cite{AKM} which is compared with the variational equation for the ``thermal equilibrium density'', which has been used recently, in some approaches to the Coulomb gas.

\subsubsection*{Notational remark} For $z=x+iy$ we write $\d=\frac 1 2 (\d_x+\frac 1 i \d_y)$ and $\dbar=\frac 1 2 (\d_x-\frac 1 i \d_y)$ for the usual complex derivatives with respect to $z$. Thus $\d\dbar=\frac 1 4(\d_{xx}+\d_{yy})$ is the normalized Laplacian (often denoted by $\Delta Q$); for the Hele-Shaw potential \eqref{hsdef} we have $\Delta =\d\dbar Q$ in the interior of $\Sigma$.

\subsubsection*{Plan of this paper} We begin in Section \ref{OCest} and \ref{Prelb} by reinvestigating the techniques of Thoma \cite{Th} transplanted to the present context. As we shall here by concerned with more precise quantitative bounds (such as $\Delta$-dependence of estimates), and since the notations and normalizations are quite different, we shall
 require a new analysis. 

In Section \ref{local} we recall some notions pertaining to weighted polynomials \cite{A} and prove the upper bound in Theorem \ref{loca}.

In Section \ref{locals} we prove the equicontinuity in Theorem \ref{mth}.

Finally, in Section \ref{concrem} we give some concluding remarks concerning Ward's equation and the relationship between the one-point function and the thermal equilibrium density.

\section{Overcrowding estimates} \label{OCest}

The goal of this section is to formulate and prove a suitable variant of Thoma's overcrowding estimate in \cite[Theorem 1]{Th}.

It is convenient to fix a small $\eta>0$ and put
\begin{equation}\label{sigman}\Sigma_\eta=\{z\in\Sigma\,;\,\dist(z,\d \Sigma)> \eta\}.\end{equation}
We assume that $\eta$ is small enough that the droplet $S$ is contained in $\Sigma_{2\eta}$.

It is intuitively helpful to recall that, by \cite[Theorem 2]{A}, a random sample $\{z_j\}_1^n$ from \eqref{bgm} is contained in $\Sigma_\eta$ with overwhelming probability if $n$ is large. (However, we shall not need this observation in what follows.)

We shall initially consider the Coulomb gas as a random point $(z_j)_1^n\in\C^n$, picked with respect to $\Prob_n^\beta$. (The more common perspective of random configurations $\{z_j\}_1^n$ is obtained by
introducing suitable combinatorial factors, as we will do later on.)

\subsection{Convolution operators} \label{convop} Let $\mu$ a radially symmetric probability measure on $\C$ and $I\subset \{1,\ldots,n\}$ a fixed index set.

Write $\C^I$ for the set of points $w^I=(w_j)_{j\in I}$ with $w_j\in \C$. We sometimes regard $\C^I$ as a subspace of $\C^n$ by
the obvious embedding, i.e., we put $w_j=0$ for all $j\not\in I$. We write $d\mu_I(w^I)=\prod_{j\in I}d\mu(w_j)$ for the product measure
on $\C^I$.

For suitable functions $F(z_1,\ldots,z_n)$
we form $T_{I,\mu}F$ by convolving with respect to $\mu$ in the
variables $z_i$ with $i\in I$, i.e.,
$$T_{I,\mu}F(z_1,\ldots,z_n):=\int_{\C^I}F(z+w^I)\,d\mu_I(w^I).$$

We extend $T_{I,\mu}$ to a bounded self-adjoint operator on $L^2(\C^n):=L^2(\C^n,dA_n)$.

For $z$  in a given subset $D$ of $\C^n$, we have the trivial bound
$$e^{-\beta  H_n(z)}\le e^{-\beta C}e^{-\beta  T_{I,\mu}H_n(z)}$$
where
$$C=\inf\{H_n(z)-T_{I,\mu}H_n(z)\,;\, z\in D\}.$$

In the case when $D=\supp F$, we write $C_{I,\mu,F}$ for the $C$ above.

Now, by Jensen's inequality we have,
$$e^{-\beta   T_{I,\mu} H_n(z)}\le T_{I,\mu} [e^{- \beta  H_n}](z).$$
We deduce that if $F\ge 0$ is an element of $L^2(\C^n)$, then
\begin{align*}\E_n^\beta &[F(z_1,\ldots,z_n)]=\frac 1 {Z_n^\beta}\int_{\supp F}Fe^{-\beta  H_n}\, dA_n\\
&\le e^{-\beta C_{I,\mu,F}}\frac 1 {Z_n^\beta}\int_{\supp F}F\cdot T_{I,\mu}[e^{-\beta H_n}]\, dA_n\\
&= e^{-\beta C_{I,\mu,F}}\frac 1 {Z_n^\beta}\int_{\C^n}[T_{I,\mu} F]e^{-\beta H_n}\, dA_n=e^{-\beta C_{I,\mu,F}}\E_n^\beta[T_{I,\mu} F].
\end{align*}

We have proven:

\begin{lem} \label{mc} If $F\in L^2(\C^n)$, $F\ge 0$, and if $\{z_j\}_1^n$ is a random sample with respect to \eqref{bgm} then
$$\E_n^\beta[F(z_1,\ldots,z_n)]\le e^{-\beta C_{I,\mu,F}}\E_n^\beta[(T_{I,\mu}F)(z_1,\ldots,z_n)].$$
\end{lem}

\subsection{Superharmonicity estimate} In the following we will often use the following assumption for a given index set $I\subset\{1,\ldots,n\}$,
\begin{equation}\label{iass}z_i\in \Sigma_{\eta/2}\qquad \text{for all}\qquad i\in I.\end{equation}
Denote by $\# I$ the number of elements of $I$. Also suppose that the probability measure $\mu$ is supported in the small disc $D_{\eta/2}=\{z\in\C\,;\,|z|<\eta/2\}$.

We shall use the superharmonicity of the kernel $\log (1/|z|)$ to obtain bounds on the Hamiltonian $H_n$ defined in \eqref{hamn}. For the potential \eqref{hsdef} and $z\in \C^n$ satisfying \eqref{iass}, we have
\begin{align*}T_{I,\mu}&[H_n](z)\le \sum_{i\ne j}\log \frac 1 {|z_i-z_j|}+n\sum_{l\not\in I}Q(z_l)\\
&+n\sum_{k\in I}\int_{\C}(\Delta\cdot (|z_k|^2+2\re(z_k\bar{w})+|w|^2)+h(z_k+w))\, d\mu(w).\end{align*}

Also by \eqref{iass} and the fact that $h$ is harmonic on $\Sigma_{\eta/2}+D_{\eta/2}$, we have for all $k\in I$,
$$Q(z_k)=\int_\C (\Delta\cdot (|z_k|^2+2\re(z_k\bar{w}))+ h(z_k+w))\, d\mu(w).$$

We obtain the following result.

\begin{lem}\label{triv} With $Q_{0}(w):=\Delta\cdot |w|^2$, we have for all $z\in\C^n$ and $I\subset\{1,\ldots,n\}$ satisfying \eqref{iass},
that if $\supp\mu\subset \{|z|<\eta/2\}$ then
$$T_{I,\mu}[H_n](z)\le H_n(z)+n\mu(Q_{0})\cdot \# I.$$
\end{lem}

\subsection{Improved bounds} In what follows we fix an arbitrary point $p\in\Sigma_\eta$.

We will use the following notation for closed discs centered at $p$ and annuli centered at the origin
$$D(p,r):=\{z\,;\, |z-p|\le r\},\qquad A(a,b):=\{z\,;\, a\le |z|\le b\}.$$

For $0<r<\frac 1 {10}R$ we let $\mu=\mu_{r,R}$ be the uniform probability measure on the annulus
$$A_{r,R}:=A(\frac R 2 ,R-2r).$$
We assume throughout that $R$ is small: $R<\eta/2$.

Thus $\mu$ has density $\frac 1 {(R-2r)^2-R^2/4}\1_{A_{r,R}}$ with respect to $dA$, and so (with $Q_0=\Delta|w|^2$)
\begin{align*}\mu(Q_0)&=2\Delta \int_{R/2}^{R-2r} \frac {s^3}{(R-2r)^2-R^2/4}\, ds=\frac {\Delta} 2 ((R-2r)^2+(R/2)^2)<\Delta R^2.\end{align*}

We consider $z=(z_j)_1^n$ as an ordered random sample from $\P_n^\beta$ and we denote by $I(p,r,z)$ the random set $I$ of indices $i$ such that $z_i$ falls in the disk $D(p,r)$.

Next define a random variable $X_{n,p}$ by
$$X_{n,p}(z):=T_{I(p,r,z),\mu_{r,R}}[H_n](z).$$

\begin{lem} \label{l1} Fix $p\in\Sigma_{\eta}$ and denote by $$N=N_n(p,r):=\# I(p,r,z)$$ the number of indices $i\in\{1,\ldots,n\}$ such that $z_i\in D(p,r)$.
Then for all $z\in \C^n$,
$$X_{n,p}(z)\le H_n(z)+n\Delta R^2N-N(N-1)\log\frac R {4r}.$$
\end{lem}

\begin{proof} We first note that with $I=I(p,r,z)$ the assumption \eqref{iass} holds, since $p\in \Sigma_\eta$ and $r<R<\eta/2$.

The logarithmic potential
$$U^{\mu}(\zeta)=\int\log\frac 1 {|\zeta-w|}\, d\mu(w)$$
is radially symmetric and non-increasing in the outwards radial direction, and moreover it is constant on the disc $D(0,\frac R {2})$. It follows that
$$U^{\mu}(\zeta)\le U^\mu(0) \leq \sup_{w\in A_{r,R}} \log \frac {1}{w} = \log\frac{2}{R},$$
and hence
$$T_{I,\mu}[\sum_{i,j\in I,\, i\ne j}\log \frac 1 {|z_i-z_j|}]\le N(N-1)\log \frac {2} R.$$

Hence, since $\log\frac 1 {|z_i-z_j|}\ge \log \frac 1 {2r}$ when both $z_i$ and $z_j$ are in $D(p,r)$,
\begin{align*}T_{I,\mu}[\sum_{i\ne j}\log \frac 1 {|z_i-z_j|}]&\le \sum_{i\ne j}\log \frac 1 {|z_i-z_j|}-\sum_{i,j\in I\, i\ne j}\log \frac 1 {|z_i-z_j|}+N(N-1)\log \frac {2} R\\
&\le \sum_{i\ne j}\log \frac 1 {|z_i-z_j|}-N(N-1)\log \frac R {4r}.
\end{align*}
The result follows by combining with Lemma \ref{triv} and using the bound $\mu(Q_0)\le \Delta R^2$.
\end{proof}

Fix $p\in\Sigma_\eta$ and
let $I_1$ and $I_2$ be two fixed index sets with $I_1\subset I_2$.

We define a random variable $F_{I_1,I_2}$ by
$$F_{I_1,I_2}(z):=\1(\{I(p,r,z)=I_1\}\cap\{I(p,R,z)=I_2\}).$$

\begin{lem} \label{l2} Let $N_1=\# I_1$ be the number of elements of $I_1$. Then there exists a constant $C$ such that for all $z\in \C^n$,
$$T_{I_1,\mu_{r,R}}[F_{I_1,I_2}](z)\le e^{CN_1}\left(\frac r R\right)^{2N_1}\1(\{I(p,R,z)=I_2\}).$$
\end{lem}

\begin{proof}

Let $\rho$ be the density of $\mu_{r,R}$ with respect to $dA$, i.e., $\rho=\frac 1 {(R-2r)^2-R^2/4}\1_{A_{r,R}}$. Clearly $\rho$ satisfies a bound of the form
$\|\rho\|_\infty\le e^CR^{-2}$ with a suitable constant $C$.
We also have the $L^1$-bound
$$\|F_{I_1,I_2}\|_{L^1(\C^{I_1})}\le \int_{D(p,r)\cap \C^{I_1}} \, dA_{I_1}=r^{2N_1}.$$

Hence by Young's inequality, we have the pointwise upper bound
\begin{align*}T_{I_1,\mu_{r,R}}[F_{I_1,I_2}]&\le \|\rho\|_\infty^{N_1}\|F_{I_1,I_2}\|_{L^1(\C^{I_1})}\le e^{CN_1}\left(\frac r R\right)^{2N_1}.
\end{align*}

Next note that if there is an index $i\in I_2\setminus I(p,R,z)$ then $z_i\not \in D(p,R)$. If $i\not \in I_1$ then $T_{I_1,\mu_{r,R}}$ does not integrate in the $i$:th coordinate, so
$T_{I_1,\mu_{r,R}}[F_{I_1,I_2}](z)=F_{I_1,I_2}(z)=0$. If $i\in I_1$ then $T_{I_1,\mu_{r,R}}$ does integrate in the $i$:th coordinate, but only over the disc $D(z_i,R-2r)$, which is disjoint from $D(p,r)$.
Thus $F_{I_1,I_2}(z+w)=0$ whenever $w_i\in\supp \mu_{r,R}$, and again $T_{I_1,\mu_{r,R}}[F_{I_1,I_2}](z)=0$.

Finally, if there is an index $i\in I(p,R,z)\setminus I_2$ then $i\in I(p,R,z)\setminus I_1$ and $T_{I_1,\mu_{r,R}}$ does not integrate in the
$i$:th coordinate. Hence again $T_{I_1,\mu_{r,R}}[F_{I_1,I_2}](z)=0$.
\end{proof}

\subsection{Overcrowding} Continuing in the spirit of \cite{Th}, we now prove the following lemma.

\begin{lem} \label{l3} Fix $p\in\Sigma_{\eta}$ and write $\lambda=R/r\ge 10$ where $0<r<\frac 1 {10}R$ and $R<\eta/2$. There exists a constant $C$ such that for each integer $M$ satisfying
$$M\ge C\frac {n\Delta \lambda^2r^2+1/\beta}{\log(\lambda/4)},$$
we have that, with $N(p,r)=\# I(p,r,z)$,
\begin{equation}\label{get}\Prob_n^\beta(\{  N(p,r)\ge M\})\le \Prob_n^\beta(\{N(p,R)\ge \lambda^2 M\})+ e^{C'(1+\beta n\Delta \lambda^2 r^2)M-\beta M(M-1)\log(\lambda/4)},\end{equation}
where $C',c>0$ are suitable real constants.
\end{lem}

\begin{proof} Let $z=(z_j)_1^n$ denote an ordered random sample and note that the random index set $I=I(p,R,z)$ satisfies the assumption \eqref{iass}.

Let $I_1$, $I_2$ be index sets, with $I_1\subset I_2$, of sizes $k$ and $l$ respectively.

Write $I_{1,2}$ for the event
$$I_{1,2}:=\{z\,;\, I(p,r,z)=I_1,\, I(p,R,z)=I_2\}.$$

By Lemma \ref{l1} we have for all $z\in I_{1,2}$ that
$$H_n(z)-T_{I(p,r,z),\mu_{r,R}}[H_n](z)\ge -n\Delta R^2 k+k(k-1)\log \frac \lambda 4.$$

Hence by Lemma \ref{mc},
\begin{align*}\E_n^\beta([F_{I_1,I_2}])\le e^{\beta(n\Delta R^2 k-k(k-1)\log \frac \lambda 4)}\E_n^\beta(T_{I_1,\mu_{R,r}}[F_{I_1,I_2}]).\end{align*}

In view of Lemma \ref{l2}, the right hand side is bounded by
$$e^{\beta(n\Delta R^2 k-k(k-1)\log \frac \lambda 4)}e^{Ck}\lambda^{-2k}\Prob_n^\beta(\{I(p,R,z)=I_2\}),$$
with $C>0$ being the constant from Lemma \ref{l2}. We have shown that
$$\Prob_n^\beta(I_{1,2})\le e^{C_1(1+n\Delta \beta R^2)k-\beta k(k-1)\log\frac \lambda 4}\lambda^{-2k}\Prob_n^\beta(\{I(p,R,z)=I_2\}).$$

It follows that if $J_{k,l}$ is the event $\{\# I(p,r,z)=k\}\cap \{\# I(p,R,z)=l\}$, then
$$\Prob_n^\beta(J_{k,l})\le e^{C_1(1+nK \beta R^2)k-\beta k(k-1)\log\frac \lambda 4}\lambda^{-2k}{l\choose k} \Prob_n^\beta(\{N(p,R)=l\}).$$

We next observe that if $l\ge k+1$ then, by Stirling's formula,
\begin{equation}\label{fel}{l\choose k}\le 2\sqrt{\frac l {k(l-k)}}e^{k\log\frac l k+(l-k)\log\frac l {l-k}}\le e^{Ck}\left(\frac l k\right)^k,\end{equation}
for a large enough constant $C$.

In the case when $l\le \lambda^2 k$, this gives, with a new constant $C_2$,
\begin{align*}\Prob_n^\beta(J_{k,l})&\le e^{C_2(1+n\Delta\beta R^2)k-\beta k(k-1)\log\frac \lambda 4} \Prob_n^\beta(\{\# I(p,R,z)=l\}).
\end{align*}

Hence
\begin{align*}\Prob_n^\beta &(\{N(p,r)\ge M\})-\Prob_n^\beta(\{N(p,R)\ge \lambda^2 M\})\\
&\le \Prob_n^\beta(\{M\le N(p,r)\}\cap \{N(p,R)\le \lambda^2 M\})\\
\quad &=\sum_{k=M}^{\lambda^2 M}\sum_{l=k}^{\lambda^2 M}\Prob_n^\beta(J_{kl})
\le \sum_{k=M}^{\infty}e^{C_2(1+n\Delta\beta R^2)k-\beta k(k-1)\log\frac \lambda 4}.
\end{align*}

We choose $M$ large enough that for $k\ge M$, the ratio of consecutive terms in the last sum is $\le 1/2$, i.e.,
$$C_2(1+n\Delta\beta R^2)-2\beta k\log \frac \lambda 4\le \log\frac 1 2,\qquad (k\ge M).$$
This is satisfied if
\begin{equation*}M\ge \frac {C_2(1+n\Delta\beta R^2)+\log 2}{2\beta \log \frac \lambda 4}.\end{equation*}

Under this condition we obtain \eqref{get}
with a suitable $C'\ge C_2$.
\end{proof}

We now come to the main overcrowding estimate of this section; it roughly corresponds to \cite[Theorem 1]{Th}.

\begin{thm}\label{oc} Suppose $p\in \Sigma_{2\eta}$, and that $\Delta\ge \Delta_0>0$ and $\eta\ge \eta_0>0$. There exists a constant $C=C(\Delta_0,\eta_0)$ independent of $p$ such that for
 any $\lambda\ge 10$, and any $M$ satisfying
\begin{equation}\label{mcond}M\ge C\frac {\lambda^2  +1/\beta}{\log\frac \lambda 4}\end{equation}
we have
$$\Prob_n^\beta(\{N(p,1/\sqrt{n\Delta})\ge M\})\le e^{-\beta (\log \frac \lambda 4)M^2+C(1+\beta \lambda^2 )M}.$$
\end{thm}

\begin{proof}
First, given a large enough $n$, we choose an index $j_0$ such that
$\eta_0/4\le \lambda^{2j_0}/(n\Delta)< \eta_0/2$.
Choosing $C>1/(\eta_0\Delta_0)$ in \eqref{mcond} we then obtain
$\lambda^{2k_0}M>n$, and consequently
$$\Prob_n^\beta(\{N(p,\eta/4)>\lambda^{2j_0}M\})=0.$$

Set $r_0=1/\sqrt{n\Delta}$ and consider the successively larger radii $r_j=\lambda^jr_0$.

A repeated application of
Lemma \ref{l3} gives, with $M$ as in \eqref{mcond},
\begin{align}\Prob_n^\beta(\{N(p,1/\sqrt{n\Delta})\ge M\})
\le \sum_{j=0}^{j_0-1} e^{C'(1+\lambda^{2(j+1)}\beta )\lambda^{2(j+1)}M-\beta \lambda^{2(j+1)}M(\lambda^{2(j+1)}M-1)\log\frac \lambda 4}.\label{number}
\end{align}

Bounding the sum in \eqref{number} by the first term and renaming constants, we finish the proof.
\end{proof}

\section{A preliminary upper bound on the one-point function}  \label{Prelb}
We shall now prove a uniform upper bound on the one-point function $R_n^\beta(p)$, by adapting the strategy behind \cite[Theorem 2]{Th}.

For this we fix a point $p\in\Sigma_{\eta}$ (see \eqref{sigman}) and a number $r$ with $0<r<1/(4\sqrt{n\Delta})$. Here we assume that $n$ is large enough that
$1/\sqrt{n\Delta}<\eta/2$.

We define a function $F$ on $\C^n$ by (with $z=(z_1,\ldots,z_n)$)
$$F(z) = \boldsymbol{1}_{D(p,r)}(z_1).$$

We apply Lemma \ref{mc} with $\mu = \mu_n$ being the uniform probability measure on the annulus
\begin{equation}\label{Annn}A_n:=A(\frac{1}{2\sqrt{n\Delta}}, \frac{1}{\sqrt{n\Delta}}-r)\end{equation} and with the index set $I=\{1\}$ (which satisfies \eqref{iass}) to get
$$\Prob_n^\beta(\{z_1 \in D(p,r)\}) \leq e^{-\beta C_{I,\mu,F}} \E_n^\beta[T_{I,\mu}F].$$

\begin{lem}\label{pointbnd}
There is a constant $C>0$ such that for all $r\le 1/(4\sqrt{n\Delta})$,
    $$T_{I,\mu}F(z_1) \leq Cr^2n\Delta\boldsymbol{1}_{D(p,1/\sqrt{n\Delta})}(z_1).$$
\end{lem}

\begin{proof}
    Since $\mu_n = \rho_n\, dA$ and $||\rho_n||_\infty \leq Cn\Delta$ for some constant $C>0$ which is uniform for all $r\leq 1/(4\sqrt{n\Delta})$, we have by Young's inequality the  bound
    $$T_{I,\mu}F(z_1) \leq \|\rho_n\|_\infty \|F\|_{L^1(\mathbb{C})} \leq C r^2 n\Delta,\qquad (r\le \frac 1 {4\sqrt{n\Delta}}) .$$
    Moreover, the support of $F$ is only enlarged by at most a distance $\frac{1}{\sqrt{n\Delta}}-r$ after convolution by $\mu$, so that
    $\operatorname{supp} (T_{I,\mu}F) \subset \{z\,;\, z_1\in
    D(p, \frac{1}{\sqrt{n\Delta}})\}.$
\end{proof}

\begin{lem} \label{tock} We have for $0<r<1/(4\sqrt{n\Delta})$ that
    $$C_{I,\mu,F} \ge -1.$$
\end{lem}
\begin{proof}
    By Lemma \ref{triv} we have for $z\in \C^n$ with $z_1\in \Sigma_{\eta/2}$ (with $I=\{1\}$ and $Q_0(w)=\Delta\cdot|w|^2$)
    $$T_{I,\mu}[H_n](z) \leq H_n(z) + n\mu(Q_{0}),$$
    so $$C_{I,\mu,F}\ge -n\mu(Q_0).$$

On the other hand, using the notation \eqref{Annn}, we deduce that the uniform probability measure $\mu$ on $A_n$ satisfies
    \begin{align*}\mu(Q_{0})&=\Delta\int_{A_n}|w|^2\left((\frac 1{\sqrt{n\Delta}}-r)^2-\frac 1 {4n\Delta}\right)^{-1}\, dA(w)\\
    &\leq \Delta \sup_{w\in A_n} |w|^2 = \Delta \left(\frac{1}{\sqrt{n\Delta}}-r\right)^2 \leq\frac{1}{n}.
    \end{align*}
\end{proof}

The preceding lemmas show that if $0<r<1/(4\sqrt{n\Delta})$, then
\begin{equation}\label{pr}\Prob_n^\beta(\{z_1 \in D(p,r)\}) \leq Cr^2n\Delta e^\beta \Prob_n^\beta(\{z_1 \in D(p,1/\sqrt{n\Delta})\}).\end{equation}

Replacing the index $1$ by an arbitrary $k\in\{1,\ldots,n\}$ in \eqref{pr}, and
 using that the number of particles among $\{z_j\}_1^n$ in a set $D$ is $$N(D)=\sum_{k=1}^n\1_{\{z_k\in D\}},$$
we obtain
\begin{align*}
    n\Prob_n^\beta &(\{z_1 \in D(p,1/\sqrt{n\Delta}\}) = \sum_{k=1}^n \Prob_n^\beta(\{z_k \in D(p,1/\sqrt{n\Delta})\}) \\
    &= \E_n^\beta \left[N(p,\frac{1}{\sqrt{n\Delta}})\right]= \sum_{l=1}^n l\Prob_n^\beta(\{N(p,\frac{1}{\sqrt{n\Delta}}) = l\})\\
    &\qquad \leq \sum_{l=1}^\infty l\Prob_n^\beta(\{N(p,\frac{1}{\sqrt{n\Delta}}) \geq l\}).
\end{align*}
Next, from Theorem \ref{oc} there is a constant $C=C(\Delta_0,\eta_0)$ such that when $l\ge C(1+1/\beta)$,
$$\Prob_n^\beta(\{N(p,\frac{1}{\sqrt{n\Delta}}) \geq l\}) \le e^{c'(1+\beta)l-c\beta l^2},$$
where $c,c'$ are some positive constants.

This gives
\begin{align*}\sum_{l=1}^\infty l\Prob_n^\beta &(N(p,\frac{1}{\sqrt{n\Delta}}) \geq l)\le \sum_{l=1}^{\lfloor C(1+1/\beta)\rfloor}l+\sum_{l=\lceil C(1+1/\beta)\rceil}le^{c'(1+\beta)l-c\beta l^2}\\
&\le C_1 \left((1+1/\beta)^2+e^{\frac {(c')^2(1+\beta)^2}{4c\beta}}\sum_{l=\lceil C(1+1/\beta)\rceil}le^{-c\beta(\ell-\frac {c'(1+\beta)}{2c\beta})^2}\right),
\end{align*}
with a new constant $C_1$. Let us denote by $B=B(C,\beta,c,c')$ the sum in the right hand side. A simple integral estimate gives
\begin{equation}\label{gaussian}B\leq C_2\left((1+1/\beta)^2+ e^{\frac {(c')^2(1+\beta)^2}{4c\beta}}(\frac 1 {c\beta}+\frac {c'(1+\beta)}{2c\beta}\frac 1 {\sqrt{c\beta}})\right).\end{equation}

An examination of the constants $c,c'$ in Theorem \ref{oc} shows that we can take $c'=100C$ and $c=\log(5/2)$, so that the above implies
$$B\leq C_3\left( \beta^{-2}+e^{C'\beta}\right)$$
with some large constants $C'\ET{,C_3}$.

We finally have, for $r>0$ small enough,
\begin{equation}\label{fjutt}\begin{split}\E_n^\beta[N(p,r)]&=n\Prob_n^\beta(\{z_1 \in D(p,r)\})\\
&\le Cn^2\Delta r^2e^{\beta }\Prob_n^\beta(\{z_1 \in D(p,1/\sqrt{n\Delta})\})\\
&\le CB\Delta nr^2 e^{\beta}.\\\end{split}
\end{equation}

We have proven the following result, which corresponds to \cite[Theorem 2]{Th} but gives additional quantitative information (for the given class of potentials) that we will exploit below.

\begin{thm} \label{1pb} There is a constant $C$ (depending on $\Delta_0$ and $\eta_0$) such that, for all $p$ in the neighbourhood $\Sigma_\eta$ of the droplet,
$$R_n^\beta(p)\le C^\beta (1+1/\beta^2)n\Delta.$$
\end{thm}

\section{Upper bound} \label{local} We now prove Theorem \ref{loca}. Our proof combines the upper bound in Theorem \ref{1pb} with an idea from \cite[Section 3]{A}.

We start with two basic lemmas, which are adapted from \cite[Section 2]{A}.

We use the symbol $\calW_n$ to denote the linear space of weighted polynomials $f=q\cdot e^{-nQ/2}$ where $q(z)$ is a holomorphic polynomial of degree at most $n-1$.

\begin{lem} \label{silem} If $f\in\calW_n$ and $D(z,1/\sqrt{n\Delta})\subset\Sigma$
then
$$|f(z)|^{2\beta}\le n\Delta e^\beta\int_{D(z,1/\sqrt{n\Delta})}|f|^{2\beta}\, dA.$$
\end{lem}

\begin{proof} Fix $z$ as above and consider the function
$$F(w)=|f(w)|^{2\beta}e^{\Delta n\beta|z-w|^2}$$
which has, for $w\in D(z,1/\sqrt{n\Delta})$,
$$\d\dbar \log F(w)\ge 0+\Delta n\beta>0.$$

Thus $F$ is (logarithmically) subharmonic and the mean-value inequality at $w=z$ gives
$$|f(z)|^{2\beta}=F(z)\le n\Delta\int_{D(z,1/\sqrt{n\Delta})}F\, dA\le n\Delta e^\beta\int_{D(z,1/\sqrt{n\Delta})}|f|^{2\beta}\, dA.$$
\end{proof}

We also recall the following well known ``maximum principle'' (for a proof, see \cite{ST} or \cite[Lemma 2.3]{A}).

\begin{lem}\label{localem} If $f\in\calW_n$ then (with $S$ the droplet)
$$|f(z)|\le (\sup\{|f(w)|\,;\, w\in S\})\cdot e^{-n\Qeff(z)/2},$$
with $\Qeff$ being the effective potential \eqref{qeff}.
\end{lem}

With these preparations out of the way, we
fix a point $z$ in the interior of $\Sigma$ and pick $\eps>0$ small enough that the disc $W=D(z,\eps)$ is contained in the interior of $\Sigma$. We define
$\delta=\dist(S,W)$.

Let $\{z_j\}_1^n$ be a random sample with respect to \eqref{bgm} and fix an index $j$, $1\le j\le n$.

We write
\begin{equation}\label{lagp}\ell_j(w)=\prod_{i\ne j}((w-z_i)e^{-n(Q(w)-Q(z_i))/2})/\prod_{i\ne j}((z_j-z_i)e^{-n(Q(z_j)-Q(z_i))/2}),\end{equation}
for the random weighted Lagrange polynomial satisfying $\ell_j(z_k)=\delta_{jk}$. (The $\ell_j$'s are well-defined elements of $\calW_n$ with probability one).

Consider the random variables
$$Y=\int_W|\ell_1|^{2\beta}\, dA,\qquad Z=\int_{\C}|\ell_1|^{2\beta}\, dA.$$

By Lemma \ref{silem} we have for all $w\in\C$ such that $D(w,1/\sqrt{n\Delta})\subset \Sigma$ that
\begin{equation}\label{uppl}|\ell_1(w)|^{2\beta}\le
e^\beta (n\Delta) Z.\end{equation}

In particular \eqref{uppl} holds for all $w\in S$ so by Lemma \ref{localem} we have
$$|\ell_1(w)|^{2\beta}\le e^\beta (n\Delta)Ze^{-n\beta \Qeff(w)},\qquad (w\in\C).$$

Integrating the last inequality over the disc $W$ using that $|W|=\eps^2$, we find
\begin{equation}\label{wefi}Y\le e^\beta(n\Delta)\eps^2Z\max_{w\in W}\{e^{-n\beta \Qeff(w)}\}.\end{equation}

We now recall the following result, proven in \cite[Lemma 2.5]{A}.

\begin{lem} \label{lead} Let $W$ be measurable subsets of $\C$. Then
\begin{equation}\label{gott}\E_n^\beta\left[\int_W|\ell_j(z)|^{2\beta}\, dA(z)\right]=|\Sigma|\cdot p_n^\beta(W)\end{equation}
where $|\Sigma|$ is the normalized area of $\Sigma=\{Q<+\infty\}$ and where $p_n^\beta$ is the 1-point measure,
$$p_n^\beta(W):=\Prob_n^\beta(\{z_1\in W\}).$$
\end{lem}

The infinitesimal form of \eqref{gott} is
\begin{equation}\label{inff}\E_n^\beta[|\ell_j(z)|^{2\beta}]=|\Sigma|\frac 1 n R_n^\beta(z).\end{equation}

Taking expectations in \eqref{wefi}, dividing through by $\eps^2$ and sending $\eps\to 0$ we have
$$\E_n^\beta[|\ell_j(z)|^{2\beta}]\le e^{\beta}(n\Delta)e^{-n\beta \Qeff(z)}\int_\C \E_n^\beta[|\ell_j(w)|^{2\beta}]\, dA(w).$$

Using \eqref{inff} and that $\int_\C R_n^\beta\, dA=n$, we find that, for all $z\in \C$,
$$\frac {|\Sigma|} n R_n^\beta(z)\le |\Sigma|e^\beta (n\Delta)e^{-n\Qeff(z)}.$$
Combining with the bound for $z\in S$ found in Theorem \ref{1pb}, we finish the proof of Theorem \ref{loca}.
$\qed$

\begin{rem} While the pointwise expectation of $|\ell_j(z)|^{2\beta}$ is bounded as $n\to\infty$ (by \eqref{inff} and Theorem \ref{loca}) we anticipate that the maximum $\max_{1\le j\le n}\|\ell_j\|_\infty$ is almost-surely bounded as $n\to\infty$ if and only if we are in the \textit{perfect freezing regime}
$$\liminf_{n\to\infty} \frac {\beta_n}{\log n}>0.$$ See \cite{AMR,AR} for several results in this direction.
\end{rem}

\section{Equicontinuity of the $1$-point functions} \label{locals} We shall now prove Theorem \ref{mth}.

Suppose that $z,w$ are two points in the interior of $\Sigma$. We assume that the points are close enough to each other, in the sense that the
line-segment $[z,w]$ from $z$ to $w$ lies wholly in the interior of $\Sigma$, at a positive distance of at least $\eta>0$ from the boundary $\d\Sigma$.

Let $\{z_j\}_1^n$ be a random sample from \eqref{bgm} and let $\ell_j(z)$ be the random Lagrange function in \eqref{lagp}.

By the fundamental theorem of calculus,
\begin{equation}\label{labil}|\ell_j(w)|^{2\beta}-|\ell_j(z)|^{2\beta}=\int_{[z,w]} 2\beta |\ell_j(\zeta)|^{2\beta-1}\nabla (|\ell_j|)(\zeta)|\bigcdot(d\zeta),\end{equation}
where the integration is over the straight line-segment, possibly making infinitesimal detours around zeros of $\ell_j$ on that line-segment, to ensure that
$|\ell_j|$ is differentiable on the interior of the line-segment.

Let us write $A_{n\Delta}(f,\zeta)$ for the average value of a function $f$ over the disc $D(\zeta,1/\sqrt{n\Delta})$,
$$A_{n\Delta}(f,\zeta)=\fint_{D(\zeta,1/\sqrt{n\Delta})}f\, dA:=n\Delta\int_{D(\zeta,1/\sqrt{n\Delta})}f\, dA.$$

The following lemma gives a Bernstein estimate, of a similar type as in \cite[Lemma 2.3]{AR}, but with a more precise $\Delta$-dependence of the implied constant.

\begin{lem} \label{lastbil} If $f\in\calW_n$ and $f(p)\ne 0$, where $p$ has distance at least $1/\sqrt{n\Delta}$ to $\C\setminus \Sigma$, then
$$|\nabla|f|(p)|\le C\sqrt{n\Delta}A_{n\Delta}(|f|,p),$$
where $C$ is an absolute constant.
\end{lem}

\begin{proof} Set $f=q\cdot e^{-nQ/2}$.

For $p$ as above and $|z-p|<1/\sqrt{n\Delta}$ we write $H_p(z)$ for the holomorphic function
$$H_p(z)=Q(p)+2\sum_{k=1}^\infty \frac {[\d^k Q](p)}{k!}(z-p)^k.$$

By Taylor's formula,
\begin{equation}\label{t0}n|Q(z)-\re H_p(z)|\le n\Delta|z-p|^2\le 1,\end{equation}
when $|z-p|\le 1/\sqrt{n\Delta}$.

By well known computations, which can be found in \cite[proof of Lemma 2.3]{AR}, we have
$$|\nabla|f|(p)|=\left|\frac d {dz}(q(z)e^{-nH_p(z)/2})\bigm|_{z=p}\right|.$$

Using a Cauchy estimate we deduce that if $1/(2\sqrt{n\Delta})\le r\le 1/\sqrt{n\Delta}$ then
\begin{align*}\left|\frac d {dz}(q(z)e^{-nH_p(z)/2})\bigm|_{z=p}\right|&=\frac 1 {2\pi}\left|\int_{|z-p|=r}\frac {q(z)e^{-nH_p(z)/2}}{(z-p)^2}\, dz\right|\\
&\le \frac {2n\Delta} \pi\int_{|z-p|=r}|q|e^{-n\re H_p/2}\,|dz|.\end{align*}

By \eqref{t0} the last expression is dominated by
$$\frac {2n\Delta} \pi e^{1/2}\int_{|z-p|=r}|f(z)|\,|dz|.$$

Integrating the last inequality in $r$ over $1/(2\sqrt{n\Delta})\le r\le 1/\sqrt{n\Delta}$ we obtain
\begin{align*}|\nabla|f|(p)|&\le\frac {4(n\Delta)^{3/2}}\pi e^{1/2}\int_{1/(2\sqrt{n\Delta})}^{1/\sqrt{n\Delta}}dr\int_{|z-p|=r}|f(z)|\,|dz|\\
&\le 4\sqrt{en\Delta}A_{n\Delta}(|f|,p),
\end{align*}
proving the lemma with $C=4\sqrt{e}$.
\end{proof}

Armed with the lemma, we now turn to our proof of equicontinuity of the one-point functions. (We remark that the condition $\beta\ge 1/2$ has not been used above; it only enters in the following, concluding steps.)

By H\"{o}lder's inequality, we have for all $\zeta\in\C$, provided that $2\beta\ge 1$,
$$\E_n^\beta\left[|\ell_j(\zeta)|^{2\beta-1}A_{n\Delta}(|\ell_j|,\zeta)\right]\le (\E_n^\beta(|\ell_j(\zeta)|^{2\beta}))^{(1-1/2\beta)}(\E_n^\beta(A_{n\Delta}(|\ell_j|,\zeta)^{2\beta}))^{1/2\beta}.$$

Now suppose that $\zeta\in \Sigma_\eta$ where $\eta>0$ is given, and take $n$ large enough that $1/\sqrt{n\Delta}<\eta/2$.

If $2\beta\ge 1$, then by Jensen's inequality, 
\eqref{inff} and Theorem \ref{loca} (with a slightly larger $C$),
\begin{equation}\label{densiste}\E_n^\beta\left[(\fint_{D(\zeta,1/\sqrt{n\Delta})}|\ell_j|\, dA)^{2\beta}\right]\le \E_n^\beta\left[\fint_{D(\zeta,1/\sqrt{n\Delta})}|\ell_j|^{2\beta}\, dA\right]\le C^\beta\Delta|\Sigma|.\end{equation}

We conclude that
$$(\E_n^\beta(|\ell_j(z)|^{2\beta}))^{(1-1/2\beta)}(\E_n^\beta(A_{n\Delta}(|\ell_j|,\zeta)^{2\beta}))^{1/2\beta}\le C^\beta\Delta|\Sigma|.$$

Using \eqref{inff}, \eqref{labil} and Lemma \ref{lastbil}, we now estimate as follows:
\begin{align*}\frac {|\Sigma|} n |R_n^\beta(z)-R_n^\beta(w)|&=|\E_n^\beta[|\ell_j(z)|^{2\beta}-|\ell_j(w)|^{2\beta}]|\\
&\le C\beta\sqrt{n\Delta} \int_{[z,w]} |\E_n^\beta[|\ell_j(\zeta)|^{2\beta-1}A_{n\Delta}(|\ell_j|,\zeta)]|\,|d\zeta|\\
&\le C^\beta \Delta|\Sigma|\sqrt{n\Delta}|z-w|,\end{align*}
where we used \eqref{densiste} (and changed meaning of the constant $C$) to get the last inequality.
Our proof of Theorem \ref{mth} is complete. $\qed$

\section{Concluding remarks} \label{concrem}

In this section, we recall the rescaled Ward equation from \cite{AKM} and the related question of finding exact subsequential scaling limits $\rho^\beta=\lim \rho^\beta_{n_k}$. We
also compare
with the approximation by the thermal equilibrium measure, which is used in e.g.~\cite{Se} and references.

\subsection{Rescaling in Ward's identity and crystallization} In this section, we  do not assume that $Q$ is a Hele-Shaw potential, but
merely that $Q$ is lower semicontinuous and obeys the growth $Q(z)-2\log|z|\to\infty$ as $z\to\infty$ (and that $Q$ is finite on some set of positive capacity). In short, $Q/2$ is an ``admissible potential'' in the sense of \cite{ST}.

The Coulomb gas in external potential $Q$ (at inverse temperature $\beta$) is a configuration $\{z_j\}_1^n$ picked randomly with respect to the corresponding Gibbs distribution (see \eqref{bgm} and \eqref{hamn}).

To a first approximation, the Coulomb gas follows the Frostman equilibrium measure $\sigma=\sigma[Q]$ in potential $Q$, which (as before) is the unique compactly supported Borel probability measure minimizing the energy integral \eqref{wen}.

More precisely, if we assume that $Q$ is $C^2$-smooth in some neighbourhood of the droplet $S=\supp\sigma$ we have that $d\sigma=\d\dbar Q\cdot \1_S\, dA$ by Frostman's theorem \cite{ST} and a standard result implies the weak convergence
$\lim_{n\to\infty}\frac 1 n\E_n^\beta\sum_{j=1}^n f(z_j)= \sigma(f)$
for all continuous and bounded functions on $\C$, see e.g. \cite[Appendix A]{A}.

Moreover, a well-known Ward identity (see e.g. \cite{Z} and \cite{AHM,AKM}) asserts that for any suitable (e.g.~smooth and bounded) test-function $f$, the random variable
$$W_n^+[f]=\frac 1 \beta \sum_1^n \d f(z_j)-n\sum_{1}^n [f\d Q](z_j)+\frac 12 \sum_{j\ne k}\frac {f(z_j)-f(z_k)}{z_j-z_k}$$
has expectation zero:
\begin{equation}\label{wdid}\E_n^\beta[W_n^+[f]]=0.\end{equation}

When written in terms of the $1$-point function $R_n^\beta(z)$ and the $2$-point function $R_{n,2}^\beta(z,w)$, the identity \eqref{wdid} becomes
\begin{equation}\label{wdid2}\frac 1 \beta \int_\C \d f\cdot R_n^\beta\, dA-n\int_\C f\cdot \d Q\cdot R_n^\beta\, dA+\frac 1 2 \iint_{\C^2} \frac {f(z)-f(w)}{z-w}\, R_{n,2}^\beta(z,w)\, dA_2(z,w)=0.\end{equation}

It is advantageous to interpret \eqref{wdid2} as a distributional integro-differential equation for the (macroscopic) \textit{Berezin kernel}
\begin{equation}\label{nrb}\mathbf{B}_n^\beta(z,w)=\frac {R_n^\beta(z)R_n^\beta(w)-R_{n,2}^\beta(z,w)}{R_n^\beta(z)},\end{equation}
which satisfies $\mathbf{B}_n^\beta(z,z)=R_n^\beta(z)$ and $\int_\C \mathbf{B}_n^\beta(z,w)\, dA(w)=1$. This is done in sources such as \cite{AHM,AKM}.

Now fix any point $p\in\C$ such that $Q$ is smooth and strictly subharmonic at $p$, $\d\dbar Q(p)>0$.

We rescale as
in \cite{AKM}, i.e., we introduce a new ``microscopic'' variables $u$, $v$  by
\begin{equation*}z=p+\frac u {\sqrt{n\d\dbar Q(p)}},\qquad w= p+\frac v {\sqrt{n\d\dbar Q(p)}}.
\end{equation*}

In order to rescale in \eqref{wdid2}, it is convenient to introduce the rescaled one-point density and the rescaled Berezin kernel via
\begin{equation}\label{microrho}\rho_n^\beta(u)=\frac {R_n^\beta(z)}{n\d\dbar Q(p)},\qquad B_n^\beta(u,v)=\frac 1 {n\d\dbar Q(p)}\frac{R_n^\beta(z)R_n^\beta(w)-R_{n,2}^\beta(z,w)}{R_n^\beta(z)}.\end{equation}
Note that $B_n^\beta(u,u)=\rho_n^\beta(u)$.

In \cite[Theorem 7.5]{AKM} it is noted that \eqref{wdid2} gives rise to the following asymptotic relation for $B_n^\beta(u,v)$,
\begin{equation}\label{topp}\dbar_u\int \frac {B_n^\beta(u,v)}{u-v}\, dA(v)=B_n^\beta(u,u)-1-\frac 1 \beta\d_u\dbar_u \log B_n^\beta(u,u)+o(1),\end{equation}
where $o(1)\to 0$ uniformly on compact subsets of $\C$ as $n\to\infty$.

Assuming that $B_n^\beta(u,v)$ converges in an appropriate sense along some subsequence $n_k$ to an appropriate limit $B^\beta(u,v)$ (with $\rho^\beta(u):=B^\beta(u,u)>0$) we obtain the formal limiting Ward equation for
$B^\beta(u,v)$,
\begin{equation}\label{vopp}\dbar_u\int \frac {B^\beta(u,v)}{u-v}\, dA(v)=B^\beta(u,u)-1-\frac 1 \beta\d_u\dbar_u \log B^\beta(u,u).\end{equation}

While a true limiting Berezin kernel must satisfy \eqref{vopp}, there is also another natural condition, the ``integral-$1$ condition''
\begin{equation}\label{mass1}\int_\C B^\beta(u,v)\, dA(v)\equiv 1,\end{equation}
which is motivated by the fact that $\int_\C B_n^\beta(u,v)\, dA(v)=1$ for all $n$.

When $\beta=1$ the derivations of \eqref{vopp}, \eqref{mass1} are rigorous: normal families arguments and the zero-one law in \cite{AKM} ensure existence of appropriate (locally uniform) subsequential limits satisfying \eqref{buck} provided that
the limit $\rho^\beta$ does not vanish identically.

When $\beta\ne 1$, it is an open question whether it is possible to pass to a subsequential limit satisfying \eqref{vopp}, \eqref{mass1}, even when rescaling at about point in the bulk.

In what follows, we will simply assume that a nontrivial (i.e. $\rho^\beta>0$ everywhere) microscopic Berezin kernel $B^\beta$ exists along some subsequence, for a fixed $\beta$, and that $B^\beta$ satisfies \eqref{vopp} and \eqref{mass1}.

We shall now recall, from \cite[Section 7.3]{AKM}, two exact solutions to \eqref{vopp} which we term ``reference solutions'', and which are relevant for $p$ a bulk point and a regular boundary point, respectively.

\subsubsection*{Reference solution for the bulk} From \eqref{vopp}, we see that any function of the form $B^\beta(u,v)=F_\beta(|u-v|)$ is an exact solution, provided that $F_\beta(r)$ is sufficiently regular and small as $r\to\infty$. It is also natural to require that $F_\beta(0)=1$.

Among these solutions a natural candidate is the Gaussian
\begin{equation}\label{isd}B^\beta_{\refb}(u,v):=e^{-\beta|u-v|^2}.\end{equation}
This solution is indeed the correct limiting Berezin kernel when $\beta=1$, but not for $\beta\ne 1$, since $\int_\C B^\beta_{\refb}(u,v)\, dA(v)=1/\beta$ in contradiction with the integral-$1$ equation.

Now suppose that $p_n$ is a sequence of zooming-points converging as $n\to\infty$ to a limit $p$ in the bulk $\Int S$, and that $B^\beta=\lim_{k\to\infty} B^{\beta_{n_k},p_{n_k}}$ is an exact scaling limit. If $\beta\ne 1$ it
is, as far as we know, not known whether $B^\beta$ depends on the subsequence.

To put it more concretely, when $\beta>1$, the density profile near the boundary is expected to display a certain oscillatory behaviour as one moves inwards from the boundary (``crystallization''). The oscillations appear to be on the scale of the interparticle distance and decrease quickly in magnitude as one moves inwards, towards the bulk, see \cite{CFTW,CSA}. It is currently an open question whether or not these oscillations persist, to some extent, also in the bulk. If they do persist, then a bulk scaling limit should depend on whether we the zooming points $p_n$ trace a local ``peak'' or ``trough'' of the density. In other words, the question of \textit{uniqueness} of a bulk scaling limit $\{z_j\}_1^\infty$ (or ``infinite $\beta$-Ginibre ensemble'') is, to the best of our knowledge, an open matter, except when $\beta=1$.

\subsubsection*{Reference solution for the boundary}

Another exact solution to \eqref{vopp} found in \cite{AKM}, relevant at a regular boundary point $p$ 
has the appearance
\begin{equation}\label{retain}B^\beta_{\refbo}(u,v):=e^{-\beta|u-v|^2}\frac 1 2 \erfc(\sqrt{\frac \beta 2 }\re (u+v)),\end{equation}
which has ``nearly'' the behaviour one would expect at the edge, and which is in fact exactly right when $\beta=1$. However, in \cite[Section 7.3]{AKM} it is observed that $\int B^\beta_{\refbo}(u,v)\, dA(v)=1/\beta$, so this solution can again not
be a true subsequential limit when $\beta\ne 1$.

In addition, it is not hard to see that \eqref{retain} is consistent with the sum rule in \cite[Eq. (4.11)]{BF} if and only if $\beta=1$, proving in another way that \eqref{retain} is not a feasible scaling limit when $\beta\ne 1$. Incidentally, for $\beta=1$, the sum rule reduces to the ``$1/8$-formula'' in the parlance of \cite{AKM}. We thank S.-S. Byun for this observation.

However, by comparing with numerical results for the diagonal values $\rho^\beta(u)=B^\beta(u,u)$ for a true limiting $1$-point function, it can be surmised that $\rho^\beta(u)$ seems to ``oscillate'' about
$$\rho_{\refbo}^\beta(u):=\frac 1 2 \erfc(\sqrt{2\beta}\re u),$$ and the asymptotic values as $\re u\to \pm\infty$ match (see \cite{CFTW,CSA}). In this way, it seems that the reference solution \eqref{retain}  does retain some features of an actual scaling limit.

It is also worth noting that the simple proof of Gaussian field convergence in \cite{AHM,ACC}, there obtained for $\beta=1$, extends naturally to all $\beta>0$ provided that some apriori tightness and asymptotic symmetry properties for all \textit{bulk} scaling limits $B^\beta$ can be proven. In short, if after rescaling about zooming points well in the bulk,
subsequential limits $B^\beta(u,v)$ exist and are functions of $|u-v|$, then the strategy of proof in \cite{AHM} carries over to all $\beta$.
In this way, the problem is reduced to determining to what extent the (bulk) limiting measures $B^\beta(u,v)\, dA(v)$ are symmetric with respect to rotations about $u$.
We shall find opportunity to return to this question below.

\subsection{Comparison with the thermal equilibrium measure} In some approaches to the Coulomb gas (cf.~ \cite{Se} and references) the so-called thermal equilibrium measure plays a central role.
It therefore seems appropriate to compare a little with the one-point function.

For a smooth unit charge density $\delta(z)$ on the plane ($\delta\ge 0$ and $\int\delta\, dA=1$) we consider the weighted logarithmic energy
$$I_Q[\delta]=\iint_{\C^2}(\log\frac 1 {|z-w|}) \delta(z)\delta(w)\, dA(z)dA(w)+\int_\C Q\,\delta\, dA$$
and the (negative of the) entropy
$$E_Q[\delta]=\int_\C\delta\log\delta \, dA.$$
(We adopt the convention that $0\log 0=0$.)

It is natural to consider the following combined energy/entropy functional
$$F_n[\delta]=I_Q[\delta]+\frac 1 {n\beta} E_Q[\delta].$$

We let $\delta_n^\beta$ denote the minimizer among smooth $\delta\ge 0$ with $\int \delta\, dA=1$ (it exists under suitable conditions on $Q$, see \cite{Se} and references there).

The density $\delta_n^\beta$ is known as the thermal equilibrium density and it has been hypothesized that this should somehow be a better approximation of the normalized one-point function $\frac 1 n R_n^\beta$, rather
than the naive approximation by the equilibrium density $\d\dbar Q\cdot\1_S$. We shall presently show that this is not really the case.

To see this, we first note that the variational equation for $\delta_n^\beta$ is
\begin{equation}\label{del}-\delta_n^\beta+\d\dbar Q+\frac 1 {n\beta}\d\dbar\log\delta_n^\beta=0.\end{equation}

Indeed, taking a perturbation $\tilde{\delta}$ with $\int \tilde{\delta}\, dA=0$, we find
\begin{align*}0=\frac d {d\eps} F_n[\delta_n^\beta+\eps\tilde{\delta}]\biggm|_{\eps=0}&=2\int\tilde{\delta}(z)\, dA(z)\int \delta_n^\beta(w)\log\frac 1 {|z-w|}\, dA(w)\\
&+\int\tilde{\delta}(z)Q(z)\, dA(z)+\frac 1 {n\beta}\int\tilde{\delta}(z)\log(\delta_n^\beta(z))\, dA(z).\end{align*}
Eliminating $\tilde{\delta}$ this gives the pointwise equation in $z$:
$$0=2\int \delta_n^\beta(w)\log\frac 1 {|z-w|}\, dA(w)+Q(z)+\frac 1 {n\beta}\log(\delta_n^\beta(z)),$$
and applying $\d\dbar$ we immediately verify \eqref{del}.

In a way similar as before we fix a point $p$ at which $\d\dbar Q(p)>0$ and rescale the thermal equilibrium density $\delta_n^\beta$ by setting
\begin{equation}\label{microtilde}z=p+\frac 1 {\sqrt{n\d\dbar Q(p)}}u,\qquad \tilde{\rho}_n^\beta(u):=\frac {\delta_n^\beta(z)}{\d\dbar Q(p)},\end{equation}
where we pick the rescaled variable $u$ in some large but fixed compact subset of $\C$. The function $\tilde{\rho}_n^\beta$ can naturally be called a rescaled thermal equilibrium density (about $p$).

The equation \eqref{del} then leads to
\begin{equation}\label{spuck}-\tilde{\rho}_n^\beta(u)+1+\frac 1 \beta\d\dbar\log \tilde{\rho}_n^\beta(u)+o(1)=0,\end{equation}
where $o(1)\to 0$ as $n\to\infty$ uniformly for $u$ in a fixed compact subset of $\C$.

Assuming that we can form a locally uniform, \textit{strictly positive} subsequential limit
$$\tilde{\rho}^\beta=\lim \tilde{\rho}_{n_k}^\beta$$
we thus obtain the formal limiting equation:
\begin{equation}\label{buck}-\tilde{\rho}^\beta+1+\frac 1 \beta\d\dbar \log \tilde{\rho}^\beta=0.\end{equation}

We recognize \eqref{buck} as a simplified version of the Ward equation \eqref{vopp}, obtained by replacing the ``Berezin term'' $\dbar_u\int_\C\frac {B^\beta(u,v)}{u-v}\, dA(v)$ by zero. This is certainly not
a correct approximation (especially near the boundary) if we replace $\tilde{\rho}^\beta$ by a limiting one-point function $\rho^\beta$.

 Indeed, if $p$ is in the bulk (i.e., the interior) of $S$, then (at least heuristically) a subsequential limiting Berezin measure $B^\beta(u,v)\, dA(v)$ might be expected to be ``nearly'' invariant under rotations about $u$. (As we noted above this is the case when $\beta=1$.)
Therefore it could be surmised that the left hand side in \eqref{vopp} should be nearly zero, and that ``$-\rho^\beta+1+\frac 1 \beta\d\dbar \log \rho^\beta\approx 0$''.

 When rescaling about a regular \textit{boundary} point $p\in\d S$, rotational symmetry breaks down: the limiting Berezin measures $B^\beta(u,v)\, dA(v)$ are not invariant under rotations even when $\beta=1$, and the left hand side in \eqref{vopp} gives a nontrivial contribution. See \cite{AKM,AKMW} for several related calculations.

We shall now compare the one-point function $R_n^{(\beta=1)}$ and the thermal equilibrium density $n\delta_n^{\beta=1}$ for the Ginibre ensemble, where $Q(z)=|z|^2$.
Recall that the corresponding droplet is the closed unit disc $S=D(0,1)$.

We rescale about the boundary point $p=1$ using the map $z=1+u/\sqrt{n}$ and consider
the rescaled $1$-point function $\rho_n(u)$ in \eqref{microrho} along with the thermal equilibrium density $\tilde{\rho}_n(u)$ in \eqref{microtilde}.

Indeed it is well-known (eg see \cite{AKM,BF}) that $\rho_n\to \rho$ locally uniformly as $n\to\infty$ where $\rho(u)=\frac 1 2 \erfc(\frac {u+\bar{u}}{\sqrt{2}})$
and a computation using $(\erfc)'(u)=-\frac 2 {\sqrt{\pi}}e^{-u^2}$ gives
$$\d\dbar\log \rho(0) =\frac {\rho\d\dbar \rho-|\d \rho|^2}{\rho^2}=\frac {0-(1/\sqrt{2\pi})^2}{(1/2)^2}=-\frac 2 {\pi}.$$

On the other hand if we assume that $\tilde{\rho}_n\to \tilde{\rho}$ locally uniformly where $\tilde{\rho}(0)=1/2$ then by the differential equation \eqref{spuck},
$$\d\dbar\log \tilde{\rho}_n(0)=\tilde{\rho}_n(0)-1+o(1)\to -\frac 1 2.$$

Assuming that $\rho\equiv \tilde{\rho}$ in a neighbourhood of $0$ we get the contradiction $\pi =4$, so indeed $\rho\ne \tilde{\rho}$.

We have shown the following statement.

\begin{thm} Suppose that $Q(z)=|z|^2$. Then there is a constant $c>0$ such that $$\liminf_{n\to\infty} \left\{ n^{-1}\cdot\|R_n^{(\beta=1)}-n\delta_n^{(\beta=1)}\|_\infty\right\}\ge c.$$
\end{thm}

As the $\erfc$-kernel is universal when rescaling about regular boundary points for a large class of $\beta=1$ ensembles, the above argument generalizes in a straightforward way beyond the Ginibre ensemble.

While the above shows that the approximation of $R_n^{(\beta=1)}$ with $n\delta_n^{(\beta=1)}$ fails spectacularly near the boundary, it might still be hoped that the approximation should be reasonable in the bulk.
We will show presently that such hopes must be abandoned.

Fix a point $p\in \operatorname{Int} S$ and recall
the bulk
approximation (e.g.~\cite[eq. (5.24)]{BF})
\begin{equation}\label{bapp}R_n^{(\beta=1)}(z)=n\d\dbar Q(z)+\frac 1 2 \d\dbar \log \d\dbar Q(z)+\bigO(n^{-1}),\qquad (n\to\infty),\end{equation}
which holds uniformly for all $z$ in some neighbourhood of $p$ provided that $Q$ is smooth and strictly subharmonic at $p$.
(See e.g.~\cite{A2} for a short proof under the assumption that $Q$ is real-analytic at $p$; the case of smooth $Q$ can be obtained by well-known adaptations.)

For what follows, we assume that $Q$  strictly subharmonic  in a neighbourhood of the droplet $S$, and that there are no shallow points, i.e., $S^*=S$.

The $\bigO(1)$-correction term $\frac 1 2 \d\dbar \log \d\dbar Q$ appearing in \eqref{bapp} is what gives rise to the formula for the expectation of smooth linear statistics supported in the bulk in \cite{AHM,ACC}. (The general formula, for $f$'s that are not supported in the bulk, also contains some boundary terms.)

If the approximation of $R_n^{(\beta=1)}$ by $n\delta_n^{(\beta=1)}$ is to be useful in the bulk, it is thus reasonable to ask that the two functions should agree on compact subsets of the bulk up to an error which is $o(1)$ as $n\to\infty$. However, if we make the ansatz
in the bulk,
$$\delta_n^{(\beta=1)}:=\d\dbar Q+\frac 1 {2n}\d\dbar\log\d\dbar Q+o(1/n),$$
and insert in the equation \eqref{del}, we arrive at
$$-\frac 1 {2n}\d\dbar\log\d\dbar Q+\frac 1 n\d\dbar\log\d\dbar Q=o(1/n),$$
which leads to $\d\dbar\log\d\dbar Q=0$, i.e., $\log \d\dbar Q$ is harmonic.

Likewise, the ansatz $\delta_n^{(\beta=1)}:=\d\dbar Q+\frac k n\d\dbar\log\d\dbar Q+o(1/n)$ (along some subsequence) leads to $k=1$.

We have shown the following result.

\begin{thm} If $p$ is a point in the bulk, then
$$\liminf_{n\to\infty}|R_n^{(\beta=1)}(p)-n\delta_n^{(\beta=1)}(p)|\ge \frac 1 2 |\d\dbar \log\d\dbar Q(p)|.$$
\end{thm}

We conclude that if $\log \d\dbar Q$ is not harmonic (essentially: if $Q$ is not a Hele-Shaw potential) then the hypothesis that $R_n^{(\beta=1)}=n\delta_n^{(\beta=1)}+o(1)$ in the bulk is not consistent with the Gaussian field convergence of linear statistics in \cite{AHM,ACC}.
And even for Hele-Shaw potentials, the approximation in the bulk is not better than approximating with the equilibrium density $n\Delta$, see \cite{A2}.

It is possible to go further and e.g.~ come up with examples where the quotient $R_n^{(\beta=1)}/(n\delta_n^{(\beta=1)})$ is unbounded as $n\to\infty$, as will be the case for instance at a shallow outpost as in \cite{ACC2}. Cf. also the recent work \cite{KLY} where the one-point function is studied near a contact point.

In conclusion, we find that the thermal equilibrium measure is not a very good approximation of the one-point density, or at least it is not much better than approximating $R_n^{(\beta=1)}$ by the classical equilibrium density $n\d\dbar Q\,\1_S$.

In any event, the study of the macroscopic Berezin kernel \eqref{nrb} and its global and microscopic limits is a central and ongoing enterprise in random normal matrix theory (i.e., $\beta=1$). 
While substantial progress has been made in recent years, there remains many important open questions, see e.g. \cite{BF} or \cite{ACC} and the references there.

\begin{rem} The similarity between Ward's (integro-differential) equation \eqref{vopp} and the simpler PDE \eqref{del} for the thermal equilibrium measure is noted in 
\cite{AB0} in the context of high temperature crossover. (We thank S.-S. Byun for making us aware of this
fact.)
\end{rem}

\end{document}